\documentclass[12pt]{amsart}

\usepackage[a4paper]{geometry}
\usepackage{amsfonts}
\usepackage{amsmath}
\usepackage{amsthm}
\usepackage{amssymb}
\usepackage{mathrsfs}
\usepackage{esint}
\usepackage{upgreek}
\usepackage{mathtools, bm}

\usepackage{color}
	\definecolor{red}{rgb}{1,0,0} 
	\definecolor{green}{rgb}{0,1,0} 
	\definecolor{blue}{rgb}{0,0,1} 
	\definecolor{darkblue}{rgb}{0,0,0.6}
	\definecolor{darkred}{rgb}{0.6,0,0}
	
\usepackage[colorlinks,
            linkcolor=blue,
            anchorcolor=darkblue,
            citecolor=darkblue
           ]{hyperref}

\newtheorem{theorem}{Theorem}[section]

\newtheorem{corollary}[theorem]{Corollary}

\newtheorem{lemma}[theorem]{Lemma}

\newtheorem{proposition}[theorem]{Proposition}

\theoremstyle{definition}
\newtheorem{definition}{Definition}

\theoremstyle{remark}
\newtheorem{remark}{Remark}[section]

\numberwithin{equation}{section}

\newcommand{\eps}{\varepsilon}  
  
\newcommand{\vf}{\varphi}
\newcommand{\al}{\alpha}
\newcommand{\be}{\beta}
\newcommand{\ga}{\gamma}
\newcommand{\de}{\delta}
\newcommand{\la}{\lambda}
\newcommand{\si}{\sigma}

\newcommand{\tht}{\theta}
\newcommand{\om}{\omega}

\newcommand{\md}{\mathrm{d}}   
\newcommand{\vd}{\,\md}
\newcommand{\me}{\mathrm{e}}   
\newcommand{\x}{\times}
\newcommand{\pd}{\partial}   

\newcommand{\p}{\ga}  

\newcommand{\R}{\mathbf{R}}   
\newcommand{\Nat}{\mathbf{N}}

\newcommand{\PS}{\Omega}  
\newcommand{\E}{\mathbb{E}}  
\newcommand{\Prob}{\mathbb{P}}  
\newcommand{\Filt}{\mathcal{F}}  
\newcommand{\BM}{w}  

\newcommand{\lap}{\Delta}
\newcommand{\loc}{\mathrm{loc}}

\newcommand{\Dom}{\mathcal{O}}   
\newcommand{\Q}{\mathcal{Q}} 
\newcommand{\norm}[1]{[\hspace{-0.2ex}] #1 [\hspace{-0.2ex}]}
\newcommand{\hnorm}[1]{[\hspace{-0.33ex}[ #1 ]\hspace{-0.33ex}]}

\newcommand{\pdist}[1]{| #1 |_{{\bf p}}}

\newcommand{\Ct}{C}  
\newcommand{\bL}{L_\om}  

\newcommand{\aij}{a^{ij}}   
\newcommand{\sik}{\si^{ik}}

\newcommand{\bi}{{b^{i}}}

\newcommand{\xdij}{{x^{i}x^{j}}}  
\newcommand{\xdi}{{x^{i}}}

\newcommand{\dij}{{ij}}  
\newcommand{\di}{{i}}
\renewcommand{\dj}{{j}}

\newcommand{\Dij}{D_{ij}}  
\newcommand{\Di}{D_{i}}

\newcommand{\pc}{{\lambda}}  

\newcommand{\cL}{L}  
\newcommand{\cM}{\varLambda}
\newcommand{\cm}{\varpi}   
\newcommand{\m}{\upkappa} 

\title[Stochastic parabolic equations]{
On the Cauchy problem for stochastic parabolic equations 
in H\"{o}lder spaces
}

\author{Kai~Du}
\author{Jiakun Liu}

\address
	{
	School of Mathematics and Applied Statistics,
	University of Wollongong,
	Wollongong, NSW 2522, AUSTRALIA}

\email{kaid@uow.edu.au {\rm (Kai Du)}, jiakunl@uow.edu.au {\rm (Jiakun Liu)}}

\date{\today}
\thanks{Research of Liu was supported by the Australian Research Council DE140101366}

\begin{document}

\begin{abstract}
In this paper, we establish a sharp $C^{2+\alpha}$-theory for stochastic partial differential equations of parabolic type
in the whole space.

\end{abstract}
\maketitle

\section{Introduction}\label{sec1}

In this paper, we consider the Cauchy problem for second-order stochastic partial differential equations (SPDEs) of the It\^o type
\begin{equation}\label{chief}
\md u = (\aij u_\xdij + \bi u_\xdi + c u + f)\vd t + (\sik u_\xdi + \nu^k u + g^k) \vd \BM^k_t,
\end{equation}
where $\{\BM^k\}$ are countable independent standard Wiener processes
defined on a filtered complete probability space $(\PS,\Filt,(\Filt_t)_{t\ge 0},\Prob)$,
the coefficients, free terms and the unknown function $u$ are all 
random fields adapted to the filtration $\Filt_t$ that is complete and right-continuous.
Equation \eqref{chief} has many practical applications such as in probability, engineering, and economics, and has been studied since long ago (see \cite{rozovskii1990stochastic}).
A well-known example of \eqref{chief} is the Zakai equation arising in the nonlinear filtering problem, see \cite{zakai1969,rozovskii1990stochastic,pardoux1991filtrage}.
{Regularity theory for equation \eqref{chief} also play a prominent role in the study of nonlinear stochastic equations, 
see \cite{walsh1986introduction,da1992stochastic,krylov1997spde,chow2014stochastic} and references therein.}

Denote the matrices $a = (\aij)$ and $\si = (\sik)$.
The following {uniform parabolic} condition is assumed throughout the paper:
\begin{equation}\label{parab}
\pc I_n + \si\si^* \le 2a \le \pc^{-1} I_n
\quad\text{ on } 
\R^n\x[0,\infty)\x\PS,
\end{equation}
where $\pc>0$ is a constant, $\si^*$ is the transposed matrix of $\si$, and $I_n$ is the $n\times n$ identity matrix.

A random field $u$ satisfying \eqref{chief} in the sense of 
Schwartz distributions is often called a weak solution of \eqref{chief}.
The regularity of weak solutions in Sobolev spaces has already been investigated by many researchers. 
Various aspects of $L^2$-theory have been obtained since 1970s, see \cite{pardoux1975these, krylov1977cauchy, rozovskii1990stochastic,da1992stochastic} among others.
A complete $L^p$-theory was established by Krylov \cite{krylov1996l_p,krylov1999analytic} in 1990s. 
By Sobolev's embedding, one then has the regularity in some proper $C^{2+\al}$-spaces, which however requires relatively high regularities of the given data.
As an \emph{open} problem {proposed by Krylov} \cite{krylov1999analytic}, 
one desires a sharp $C^{2+\alpha}$-theory in the sense that not only that for $f, g$ belonging to a proper space $\mathcal{X}$,
the solution belongs to some kind of stochastic $C^{2+\alpha}$-spaces, 
but also that every element of this stochastic space can be obtained as a solution for certain $f, g$ belonging to the same $\mathcal{X}$.

\vspace{3pt}
The purpose of this paper is to establish such a sharp $C^{2+\al}$-theory for equation \eqref{chief}.
In order to state our result, we need to define the proper H\"older space $\mathcal{X}$ and introduce a notion of solutions. 
\begin{definition}
A predictable random field $u$ is called a \emph{quasi-classical} solution of \eqref{chief} if

(i) for each $t\in(0,\infty)$, $u(\cdot,t)$ is a twice strongly differentiable function from $\R^n$ to $\bL^{\p} := L^\p(\PS,\Filt; \R)$ for some $\p\geq2$; and 

(ii) for each $x\in\R^n$, the process $u(x,\cdot)$ is stochastically continuous and satisfies the integral equation
\begin{align*}
& u(x,T_1)-u(x,T_0)\\
= &\int_{T_0}^{T_1}\bigl[\aij(x,t) u_\xdij (x,t) + \bi (x,t) u_\xdi (x,t) 
+ c(x,t) u(x,t) + f(x,t) \bigr]\vd t \\
& + \int_{T_0}^{T_1} \bigl[ \sik (x,t) u_\xdi (x,t) + \nu^k (x,t) u(x,t) + g^k (x,t) \bigr] \vd \BM^k_t
\end{align*}
almost surely (a.s.) for all $0\le T_0 < T_1<\infty$.

If furthermore, $u(\cdot,t,\om)\in C^2(\R^n)$ for each $(t,\om)\in(0,\infty)\x\PS$, then $u$ is a {\em classical} solution of \eqref{chief}.
\end{definition}

It is well known that $\bL^\p=L^\p(\PS,\Filt; \R)$ is a Banach space equipped with the norm
$\|\xi\|_{\bL^\p} := (\E |\xi|^{\p})^{1/\p}$,
where $\p \ge 2$ is a constant fixed throughout the paper.
Let $T>0$ and $\Q_T=\R^n\x(0,T)$. 
We define the $\bL^\p$-valued H\"older spaces $\Ct^{m+\al}_x(\Q_T;\bL^\p)$ and $\Ct^{m+\al,\al/2}_{x,t}(\Q_T;\bL^\p)$ as follows, 
where $\beta=(\beta_1,\dots,\beta_n)$ denotes a multi-index and $|\beta|=\beta_1 +\cdots +\beta_n$.

\begin{definition}\label{nota}
For $m \in \Nat := \{0,1,2,\dots\}$ and $\alpha\in(0,1)$, the space $\Ct^{m+\al}_x(\Q_T;\bL^\p)$ consists of {all {predictable} random fields $u : \Q_T\times\Omega \to \R$ such that
$u(\cdot,t)$ is an $\bL^\p$-valued strongly continuous function for each $t$ and}
\begin{equation}\label{norm1}
|u|_{m+\al;\Q_T}^{{\bL^\p}} := |u|_{m;\Q_T}^{{\bL^\p}}
+ \max_{|\be|= m}\sup_{t,\,x\neq y} 
\frac{\|D^\be u(x,t) - D^\be u(y,t)\|_{\bL^\p}}{|x-y|^\al} < \infty,
\end{equation}
where $|u|_{m;\Q_T}^{{\bL^\p}}=\max_{|\be|\le m}\sup_{(x,t)\in\Q_T} \|D^\be u(x,t)\|_{\bL^\p}$, and the derivatives are defined with respect to the spatial variable in the strong sense, see \cite{hille1957functional}.

Using the parabolic module
$\pdist{X} := |x| + \sqrt{|t|}$ for $X=(x,t)\in\R^n\x\R$,
we define $C^{m+\al,\al/2}_{x,t}(\Q_T;\bL^\p)$ to be the set of all $u\in C^{m+\al}_x(\Q_T;\bL^\p)$
such that
\begin{equation}\label{norm2}
|u|_{(m+\al,\al/2);\Q_T}^{{\bL^\p}} := |u|_{m;\Q_T}^{{\bL^\p}} + \max_{|\be|= m}\sup_{X\neq Y} 
\frac{\|D^\be u(X) - D^\be u(Y)\|_{\bL^\p}}{\pdist{X-Y}^\al} < \infty.
\end{equation}
\end{definition}
Similarly, we can define the norms \eqref{norm1} and \eqref{norm2} over a domain $Q=\Dom \x I$, for any domains $\Dom\subset\R^n$ and $I\subset\R$.
See \S\ref{sec2.1} for more general definitions.

\vspace{5pt}
We can now state our main result, while a detailed explanation of the coefficients and free terms has to be postponed to Assumption (H) in \S\ref{sec2.2}.
\begin{theorem}\label{thm1}
Assume that the classical $C^\al_x$-norms of $\aij, \bi,c,\si^i,\si^i_{\!x},\nu,\nu_x$ are all dominated by a constant $K$
uniformly in $(t,\om)\in(0,T)\x\PS$, and the condition \eqref{parab} is satisfied.
If $f\in \Ct^{\al}_x(\Q_T;\bL^\p)$, $g\in\Ct^{1+\al}_x(\Q_T;\bL^\p)$ for some $\p\ge 2$,
then equation~\eqref{chief} with a zero initial condition
admits a unique quasi-classical solution $u$ in $\Ct^{2+\al,\al/2}_{x,t}(\Q_T;\bL^\p)$.
\end{theorem}

The Cauchy problem with nonzero initial value can be easily reduced to our case by a simple transform.
Such an established $C^{2+\alpha}$-theory is sharp in the sense that as proposed by Krylov in \cite{krylov1999analytic}. 
We remark that by an anisotropic Kolmogorov continuity theorem (see \cite{dalang2007hitting}), 
if $\p > n/{\al}$, the above obtained quasi-classical solution $u$
has a $C^{2+\delta}_x$ modification with $\delta<\al-n/{\p}$;
and if $\p > (n+2)/{\al}$, then $u$ 
has a $C^{2+\delta,\delta/2}_{x,t}$ modification with $\delta<\al-(n+2)/{\p}$.

Our result can be applied to a wide range of nonlinear filtering problems.
For example, the Zakai equation is the homogeneous case of \eqref{chief} as the terms $f$ and $g$ vanish, and is often associated with a deterministic initial value condition. 
As an application of Theorem \ref{thm1}, we have a more general result that embracing the Zakai equation. 

\begin{corollary}\label{2stthm}
{Under the hypotheses of Theorem \ref{thm1},}
if the initial value $u(\cdot,0)\in C^{2+\al}(\R^n)$, the free terms $f\in L^{\infty}([0,\infty); C^{\al}(\R^n))$ and $g\in L^{\infty}([0,\infty); C^{1+\al}(\R^n))$ are all nonrandom,
then equation \eqref{chief} admits a unique classical solution.
\end{corollary}

Note that in Corollary \ref{2stthm}, the coefficients $a,b,c,\si,\nu$ are allowed to be random 
{and merely required to satisfy natural regularity assumptions.}
To the best of our knowledge, this is a new result concerning the classical solution of the Zakai equation.

\vspace{5pt}
The solvability of SPDEs in $\bL^{\p}$-valued H\"older spaces 
was previously studied by Rozovsky \cite{rozovskiui1975stochastic} and Mikulevicius \cite{mikulevicius2000cauchy}.
However, they both need to assume the leading coefficient $a$ is deterministic and there is no derivatives of $u$ in the stochastic term, namely $\si\equiv0$. 
Such a strong restriction excludes many interesting examples and applications. 
Moreover, neither of them addressed the time-continuity of second-order derivatives 
of $u$, which is now obtained in our Theorem \ref{thm1}.
For more related results under other appropriate assumptions, we refer the reader to, for example \cite{kunita1982stochastic,walsh1986introduction,funaki1991regularity,chow1994stochastic,bally1995approximation} and references therein. 
Most recently, Hairer \cite{Hairer} created an abstract theory of regularity structures for SPDEs including multi-level Schauder estimates. 
Our approach in this paper is totally different to that of \cite{Hairer}.

\vspace{5pt}
The solvability in Theorem \ref{thm1} can be derived by the standard method of continuity (see \S\ref{sec2.2}), once we have the following Schauder estimate. 

\begin{theorem}\label{schauder1}
Under the hypotheses of Theorem \ref{thm1}, 
letting $u$ be a quasi-classical solution of \eqref{chief} and $u(\cdot,0) = 0$,
there is a positive constant $C$ depending only on $n,\pc,\p,\al$ and $K$ such that
\begin{equation}\label{schauder.est0}
|u|_{(2+\al,\al/2);\Q_T}^{{\bL^\p}} 
\le C \me^{CT} (|f|_{\al;\Q_T}^{{\bL^\p}} + |g|_{1+\al;\Q_T}^{{\bL^\p}}).
\end{equation}
\end{theorem}

\vspace{5pt}
In non-stochastic cases, the Schauder estimate is one of the most important estimates for elliptic and parabolic equations, which was traditionally built upon the potential theory, and then was obtained via different approaches by, for instance, Campanato \cite{campanato1964proprieta}, Trudinger \cite{trudinger1986new}, Schlag \cite{Schlag}, Simon \cite{simon1997schauder}, and others.
Also, perturbation arguments were used by Safonov \cite{safonov1984classical}, Caffarelli \cite{caffarelli1989interior}, and Wang \cite{Wang2006}, which can be applied to fully nonlinear equations. 
However, each individual method of the above has some essential defect when applied to the SPDEs, partially because of the adaptedness issues, and also the absent of a proper maximum principle for the SPDEs.
\footnote{Two different types of maximum principle for SPDEs were obtained in \cite{denis2005p} and \cite{krylov2007maximum}, respectively, but neither is suitable for our circumstance.}

\vspace{5pt}
In our proof of Theorem \ref{schauder1}, we adopt the perturbation scheme from Wang's work \cite{Wang2006}, while instead of using the maximum principle, we establish certain integral-type estimates as inspired by the work of Trudinger \cite{trudinger1986new}. See also \cite{DuLiu}, where many of our results were first announced.

\vspace{5pt}
The paper is organised as follows.
In Section \ref{sec2}, we introduce some notation and extend Definition \ref{nota} to general cases in \S\ref{sec2.1}, which will be used in subsequent sections. In \S\ref{sec2.2}, by assuming having Theorem \ref{schauder1}, we prove Theorem \ref{thm1} via the method of continuity.
In Sections \ref{sec3} and \ref{sec4}, we consider a model equation
\begin{equation}
\md u = (\aij u_\xdij + f)\vd t + (\sik u_\xdi + g^k) \vd \BM^k_t,
\end{equation}
where the random coefficients $a$ and $\si$ are independent of $x$.
We first prove some auxiliary estimates in \S\ref{sec3}, and then establish the interior H\"older estimate in \S\ref{sec4}, which is the crucial ingredient of obtaining the Schauder estimate \eqref{schauder.est0}.
In Section \ref{sec5}, we prove Theorem \ref{schauder1} by establishing the global Schauder estimate for the Cauchy problem of \eqref{chief}.
Some properties and approximation of $\bL^\p$-valued continuous functions are proved in Appendix.

\section{Preliminaries}\label{sec2}

\subsection{Notation}\label{sec2.1}
For a function $u$ of $x=(x^1,\dots,x^n)\in\R^n$, we denote
\begin{gather*}
 u_\di = \Di u = u_{x^i},
 \quad u_\dij = \Dij u = u_{x^i x^j},
 \quad
 D u = u_{x} = (u_{1},\dots,u_{n}).
\end{gather*}
Hereafter, $\beta=(\beta_1,\dots,\beta_n)$ with $\beta_i\in\Nat=\{0,1,2,\dots\}$ is a multi-index; we denote
\[
D^\beta = D^{\beta_1} \cdots D^{\beta_n},
\quad |\beta|=\beta_1 +\cdots +\beta_n.
\]
For $m\in\Nat$ we denote $D^m u$ the set of all $m$-order derivatives of $u$.
These $D^m u(x)$ for each $x$ are regarded as elements
of a Euclidean space of proper dimension.

Let $\Dom$ be a domain in $\R^n$, $I\subset \R$ be an interval, and $Q:=\Dom\x I$.
Let $E$ be a Banach space.
For a function $h : \Dom\to E$, we define
\begin{align*}
& [h]^E_{0;\Dom} = |h|^E_{0;\Dom} := \sup_{x\in\Dom}\|h(x)\|_E;\quad \mbox{and} \\
& [h]^E_{\al;\Dom} := \sup_{x,y\in\Dom,\ x \neq y}\frac{\|h(x) - h(y)\|_E}{|x-y|^{\al}} \ \ \mbox{ for } \al \in (0,1).
\end{align*}
Then for $m\in\Nat$ and $\al\in (0,1)$, denote 
\begin{align*}
|h|^E_{m;\Dom} := & \max_{|\beta|\le m}|D^\beta h|^E_{0;\Dom}, \\
|h|^E_{m+\al;\Dom} := & |h|^E_{m;\Dom} + \max_{|\beta|= m}[D^\beta h]^E_{\al;\Dom}.
\end{align*}
Here and below, all the derivatives of an $E$-valued function are
defined with respect to the spatial variable in the strong sense, 
see \cite{hille1957functional}.

For a function $u: Q=\Dom \x I\to E$, we define
\vspace{-5pt}
\begin{align*}
& [u]^E_{\al;Q} := \sup\nolimits_{t\in I} [u(\cdot,t)]^E_{\al;\Dom}, \\
& |u|^E_{m+\al;Q} := \sup\nolimits_{t\in I} |u(\cdot,t)|^E_{m+\al;\Dom}. 
\end{align*}
Letting $\pdist{X} = \pdist{(x,t)} = |x| + \sqrt{|t|}$ be the parabolic modulus of $X=(x,t)\in\R^n\x\R$,
we further define 
\begin{align*}
[u]^E_{(m+\al,\al/2);Q} & :=  \max_{|\beta|=m}\sup_{X,Y\in Q,\ X\neq Y} 
\frac{\|D^\beta_x u(X) - D^\beta_x u(Y)\|_E}{\pdist{X-Y}^{\al}}, \\
|u|^E_{(m+\al,\al/2);Q} & := |u|^E_{m;Q} + [u]^E_{(m+\al,\al/2);Q}.
\end{align*}

In the following context, the space $E$ is either a Euclidean space or $\bL^\p$, where $\p\in [2,\infty)$ is a fixed constant.
We omit the superscript when $E$ is a Euclidean space.
In the case of $E=\bL^\p$, we introduce some new notation: 
\begin{equation*}
\norm{\cdot}_{\dots} := |\cdot|^{\bL^\p}_{\dots},
\quad
\hnorm{\,\cdot\,}_{\dots} := [\,\cdot\,]^{\bL^\p}_{\dots}.
\end{equation*}
For instance, 
$ \norm{u}_{m+\al;Q} = |u|^{\bL^\p}_{m+\al;Q}$, and
$\hnorm{u}_{(\al,\al/2);Q} = [u]^{\bL^\p}_{(\al,\al/2);Q}$.

Using the above notation, the spaces $C^{m+\al}_x(Q;\bL^\p)$ and $C^{m+\al,\al/2}_{x,t}(Q;\bL^\p)$ defined in Definition \ref{nota} are the sets of all predictable random fields $u:Q\x\PS\to\R$ 
such that $\norm{u}_{m+\al;Q}$ and $\norm{u}_{(m+\al,\al/2);Q}$ are finite, respectively.

\subsection{The solvability}\label{sec2.2}

Let $\cL$ and $\cM^k$ be differential operators
\begin{equation*}
\cL = \aij \Dij + \bi \Di + c,
\quad \mbox{and }\ \cM^k = \sik \Di + \nu^k.
\end{equation*}
The Cauchy problem under consideration can be written as
\begin{equation}\label{cauchy.pr}
\begin{aligned}
& \md u = (\cL u + f)\vd t + (\cM^k u + g^k) \vd \BM^k_t 
\quad \text{in } \Q:= \R^n\x [0,\infty),\\
& u(\cdot,0) = 0 \quad \text{in } \R^n.
\end{aligned}
\end{equation}

Throughout the paper, we assume that
\medskip

({\bf H})~
For all $i,j=1,\dots,n$, 
the random fields $\aij$, $\bi$, $c$ and $f$ are real-valued, 
and $\si^i$, $\nu$ and $g$ are $\ell^2$-valued;
all of them 
are {predictable}.
$\aij$ and $\si^i$ satisfy the {stochastic parabolic condition} \eqref{parab}.
For some $\al\in(0,1)$ there exists a constant $K$ such that
$\max\{|\aij|_{\al;\Q},|\bi|_{\al;\Q} ,|c|_{\al;\Q},
|\si^{i}|_{1+\al;\Q} , |\nu|_{1+\al;\Q} \}
\le K$ for all $\om\in\PS$.

\medskip

Recall that
$\Q_T = \R^n \x (0,T)$, and $T>0$.
Using the notation in \S\ref{sec2.1}, the Schauder estimate in Theorem \ref{schauder1} can be written as: 
There is a positive constant $C$ depending only on $n,\pc,\p,\al$ and $K$, such that
\begin{equation}\label{schauder.est}
\norm{u}_{(2+\al,\al/2);\Q_T} 
\le C \me^{CT} (\norm{f}_{\al;\Q_T} + \norm{g}_{1+\al;\Q_T})
\end{equation}
for any $T>0$, where $u$ is a quasi-classical solution of the Cauchy problem \eqref{cauchy.pr}.

\begin{proof}[Proof of Theorem~\ref{thm1}]
With the above a priori estimates in hand,
we can obtain the solvability of the Cauchy problem \eqref{cauchy.pr}
by the method of continuity.
Consider
\begin{equation}\label{con.eq}
\md u = (\cL_s u + f)\vd t + (\cM^k_s u + g^k) \vd \BM^k_t,
\quad
u(\cdot,0) = 0,
\end{equation}
where $s\in [0,1]$ and
\[
\cL_s := s\cL + (1-s)\lap,
\quad
\cM^k_s := s \cM^k.
\]
Evidently, the solutions of \eqref{con.eq} satisfy 
the a priori estimate \eqref{schauder.est} with the constant $C$ independent of $s$. 
In view of \cite[Theorem~5.2]{gilbarg2001elliptic},
it suffices to show the solvability of the stochastic heat equation (the case $s=0$):
\begin{equation}\label{s=0}
\md u = (\lap u + f)\vd t +  g^k \vd \BM^k_t,
\quad
u(\cdot,0) = 0.
\end{equation}

Set $f^\eps = \vf^\eps * f$ and $g^\eps = \vf^\eps * g$,
where $\vf^\eps(x)=\eps^n\vf(x/\eps)$ and $\vf$ is a nonnegative and symmetric mollifier defined on $\R^n$ (see Appendix).
Then (from Lemma~\ref{lemA1.5} in Appendix) we have that $f^\eps\in \Ct^{\al}_x(\Q;\bL^\p)$ and $g^\eps\in\Ct^{1+\al}_x(\Q;\bL^\p)$ satisfying \begin{equation}\label{approxam}
\norm{f^\eps - f}_{\al/2;\Q_T} + \norm{g^\eps - g}_{1+\al/2;\Q_T} 
\to 0,\quad\text{as } \eps\to 0.
\end{equation}
Moreover, $f^\eps(x,t,\om)$ and $g^\eps(x,t,\om)$ are smooth in $x$ for any $(t,\om)$, and 
$f^\eps,g^\eps\in C^m(\Q_T;\bL^\p)$ for all $m\in\Nat$,
so by Fubini's theorem,
\begin{align*}
& \E\int_{\Q_T} (1+|x|^2)^{-p} (|D^m f^\eps (x,t)|^\p + |D^m g^\eps (x,t)|^\p)\vd x \md t \\
& \le \int_{\Q_T} (1+|x|^2)^{-p} \bigl(\E|D^m f^\eps (x,t)|^\p + \E|D^m g^\eps (x,t)|^\p \bigr)\vd x \md t \\
& \le C(n,p)T (\norm{f^\eps}_{m;\Q_T}^\p + \norm{g^\eps}_{m;\Q_T}^\p)<\infty \quad\forall\,m\in\Nat
\end{align*}
with $2p > n$.
Therefore, it follows from Krylov--Rozovsky \cite[Theorem~2.2]{krylov1982stochastic}
that \eqref{s=0} with free terms $f^\eps$ and $g^\eps$
admits a unique weak solution $u^\eps$ satisfying 
\[
\E \sup_{t\in[0,T]}\int_{\R^n} (1+|x|^2)^{-p} |D^m u^\eps(x,t)|^\p 
\vd x < \infty\quad\forall\,m\in\Nat,
\]
and by Sobolev's embedding, $u^\eps$ is smooth in $x$, and
$\E|D^m u^\eps(x,t)|^\p <\infty$ for each $(x,t)\in\Q_T$ and $m\in\Nat$ (see Lemma~\ref{lemA1} in Appendix).
From estimate~\eqref{schauder.est} (with $\alpha$ instead of $\alpha/2$) and keeping \eqref{approxam} in mind, we have
\[
\norm{u^{\eps} - u^{\eps'}}_{2;\Q_T} \le C (\norm{f^{\eps} - f^{\eps'}}_{\al/2;\Q_T} + \norm{g^{\eps} - g^{\eps'}}_{1+\al/2;\Q_T}) \to 0
\]
as $\eps,\eps' \to 0$.
Hence, $u^\eps$ converges to a function 
$u \in \Ct^{2,0}_{x,t}(\Q_T;\bL^\p)$ that apparently solves~\eqref{s=0}.
The regularity and the uniqueness follow directly from the estimate~\eqref{schauder.est}.
\end{proof}

\section{Auxiliary estimates for the model equation}\label{sec3}

In order to prove Theorem \ref{schauder1}, we start by considering the model equation
\begin{equation}\label{model}
\md u = (\aij u_\dij + f )\vd t + (\sik u_\di + g^k) \vd \BM^k_t,
\end{equation}
where $\aij$ and $\sik$ are predictable processes,
independent of $x$,
and satisfy the stochastic parabolic condition \eqref{parab}.
We first prove some auxiliary estimates for the model equation in this section, and then proceed to the interior H\"older estimate in the next section.

In the following two sections we mainly concern the local estimates for the equation \eqref{model}, so we can only focus on the estimates around the origin on account of translation.
Indeed, we can reduce the estimates around a point $(x_0,t_0)$ to the estimates around the origin by use of the change of variables $(x,t)\mapsto (x-x_0,t-t_0)$.
For this reason, {\em we may consider \eqref{model} in the entire space $\R^n\times\R$.}
On the other hand, if $u$ satisfies \eqref{model} in $\R^n\times[0,\infty)$ with $u(x,0)=0$, the zero extensions of $u$, $f$ and $g$ 
(i.e., these function are defined to all equal zero for $t<0$) satisfy the equation in the entire space, 
where the extension of coefficients and Wiener processes are quite easy; for example,
we can define $\aij(t) = \delta^{ij}$ and $\sik(t)=0$ for $t<0$, and $\BM_t := \tilde{\BM}_{-t}$ for $t<0$ with $\tilde{\BM}$ being an independent copy of $\BM$.

Let $\Dom \subset \R^n$, and $H^m(\Dom)=W^{m,2}(\Dom)$ be the usual Sobolev spaces.
Let $I\subset\R$ and $Q=\Dom\x I$.
Define
\begin{align*}
L^p_\om L^q_t H^{m}_{x}(Q) & := L^{p}(\PS; L^q (I; H^{m}(\Dom))),
\quad \mbox{for }\ p,q \in [1,\infty].
\end{align*}
For $r>0$, we denote
\begin{equation}\label{BQ}
B_r(x) = \{y\in\R^n:|y-x|<r\},
\quad
Q_r(x,t) = B_r(x) \x (t-r^2,t],
\end{equation}
and simply write $B_r = B_r(0)$, $Q_r = Q_r(0,0)$.

\begin{proposition}\label{local}
Let $m$ be a positive integer, $r \in (0,\infty)$ and $\tht \in (0,1)$.
Let $u\in L^p_\om L^2_t H^{m+1}_{x}(Q_{r})$ solve~\eqref{model} in $Q_r$ with 
$f\in L^p_\om L^2_t H^{m-1}_{x}(Q_{r})$ 
and $g\in L^p_\om L^2_t H^{m}_{x}(Q_{r})$.
Then there exists a constant $C=C(n,p,\pc,m,\tht)$ such that
\begin{gather*}
\|D^m u\|_{L^p_\om L^\infty_t L^2_{x}(Q_{\tht r})}
+ \|D^m u_{x}\|_{L^p_\om L^2_t L^2_{x}(Q_{\tht r})}
\le C r^{-m-1}\|u\|_{L^p_\om L^2_{t}L^2_{x}(Q_{r})} \\
+ C \sum_{k=0}^{m-1} r^{-m+k+1} \|D^{k} f\|_{L^p_\om L^2_t L^2_{x}(Q_{r})} 
+ C \sum_{k=0}^{m} r^{-m+k} \|D^k g\|_{L^p_\om L^2_t L^2_{x}(Q_{r})}.
\end{gather*}
Consequently, for $2(m-|\be|)>n$,
\begin{gather*}
\| \sup\nolimits_{Q_{\tht r}} \! |D^\be u| \|_{L^{p}_{\om}}
\le  
C r^{-|\be| - n/2 - 1} \|u\|_{L^p_\om L^2_{t}L^2_{x}(Q_{r})} \\
+ C \sum_{k=0}^{m-1} r^{-|\beta|-n/2+k+1} \|D^{k} f\|_{L^p_\om L^2_t L^2_{x}(Q_{r})} 
+ C \sum_{k=0}^{m} r^{-|\beta|-n/2+k} \|D^k g\|_{L^p_\om L^2_t L^2_{x}(Q_{r})}.
\end{gather*}
\end{proposition}

It suffices to prove the first inequality as
the second one follows directly from Sobolev's embedding theorem.
In fact, from Sobolev's embedding theorem, one has $H^{m}_{x}(Q_{\tht r}) \subset C^j_x(Q_{\tht r})$
if $2(m-j)>n$, see \cite[Theorem~4.12]{adams2003sobolev}. 
Hence, $L^p_\om L^\infty_t H^m_{x}(Q_{\tht r}) \subset L^p_\om L^\infty_t C^j_{x}(Q_{\tht r})$.
More specifically, if $2(m-|\be|)>n$, then
\[
\| \sup\nolimits_{Q_{\tht r}} \! |D^\be u| \|_{L^{p}_{\om}}
\le C \|u\|_{L^p_\om L^\infty_t H^m_{x}(Q_{\tht r})},
\]
for a constant $C$ independent of $r$.
In order to establish the above local estimates, we first show the following
mixed-norm estimates for the model equation \eqref{model}.

\begin{lemma}\label{p2m}
Let $\Q_T=\R^n\x [0,T]$,
$p\ge 2$ and $m \in \Nat$.
Suppose $f\in L^p_\om L^2_t H^{m-1}_{x}$ and $g\in L^p_\om L^2_t H^{m}_{x}$.
Then equation~\eqref{model} with zero initial condition admits a unique weak solution 
$u \in L^p_\om L^\infty_t H^{m}_{x}(\Q_T) \cap L^p_\om L^2_t H^{m+1}_{x}(\Q_T)$,
and for any multi-index $\be$ such that $|\be| \le m$,
\begin{equation}\label{p22est}
\|D^\be u\|_{L^p_\om L^\infty_t L^2_{x}}
+ \|D^\be u_{x}\|_{L^p_\om L^2_t L^2_{x}} 
\le C (\|D^\be f\|_{L^p_\om L^2_t H^{-1}_{x}} 
+ \|D^\be g\|_{L^p_\om L^2_t L^2_{x}}).
\end{equation}
where $C = C(n,p,T,\pc)$.
\end{lemma}

\begin{proof}
The special case of $p=2$ follows from the $L^2$-theory of SPDEs, for instance, 
see \cite[Theorem~4.1.2]{rozovskii1990stochastic}.
We prove the general cases of $p\geq2$ by induction of $m$. 
Thus it suffices to show the estimate \eqref{p22est} for $m = 0$,
namely
\begin{equation}\label{p22est0}
\|u\|_{L^p_\om L^\infty_t L^2_{x}}
+ \|u\|_{L^p_\om L^2_t H^1_{x}} 
\le C (\|f\|_{L^p_\om L^2_t H^{-1}_{x}} 
+ \|g\|_{L^p_\om L^2_t L^2_{x}}).
\end{equation}

Take a stopping time $\tau:\PS \to [0,T]$ such that
\begin{equation}\label{pre.ass}
\E\biggl[
\biggl( \sup_{t\in[0,\tau]} \int_{\R^n} |u(x,t)|^2 
\vd x 
+ \int_0^\tau\!\!\int_{\R^n} |u_x(x,t)|^2 \vd x \md t
\biggr)^{\!\frac{p}{2}}
\bigg]
< \infty.
\end{equation}
By the It\^o formula (cf.~\cite[Theorem~4.2.2]{rozovskii1990stochastic}) and integration by parts,
\begin{align*}
\|u(t)\|^2_{L^2_x} 
& = \int_0^t\!\!\int_{\R^n} \Bigl[
-2\aij u_\di u_\dj
+ \|\si^{i} u_\di + g\|_{\ell^2}^2 \Bigr] \vd x \md s \\
&\quad + 2\int_0^t \langle u(t), f(t) \rangle \vd t
+ 2\int_0^t\!\!\int_{\R^n} ug^k \vd x \md \BM^k_s,
\end{align*}
where $\langle \cdot, \cdot\rangle$ denotes the duality product between and $H^1_x$ and $H^{-1}_x$.
By the parabolic condition~\eqref{parab} and using Young's inequality, we have
\begin{align*}
& \sup_{t\in[0,\tau]}\|u(t)\|^2_{L^2_x} + \int_0^\tau \|u(t)\|^2_{H^1_x} \vd t \\
\nonumber & \le C \int_0^\tau \bigl[
\|f(t)\|^2_{H^{-1}_x} + \|g(t)\|^2_{L^2_x}\bigr] \md t
+ C \left\vert \sup_{t\in[0,\tau]}
\int_0^t\!\!\int_{\R^n} ug^k \vd x \md \BM^k_s \right\vert.
\end{align*}
Then computing $\E[\,\cdot\,]^{p/2}$ 
on both sides gives us that
\begin{align} \label{eq301}
& \E\biggl[
\biggl( \sup_{t\in[0,\tau]}\|u(t)\|^2_{L^2_x} + \int_0^\tau \|u(t)\|^2_{H^1_x} \vd t
\biggr)^{\!\frac{p}{2}} \biggr]\\
\nonumber & \le C \bigl(\|f\|^p_{L^p_\om L^2_t H^{-1}_{x}(\Q_T)} 
+ \|g\|^p_{L^p_\om L^2_t L^2_{x}(\Q_T)} \bigr)
+ C \E\biggl[\sup_{t\in[0,\tau]}
\left\vert\int_0^t\!\!\int_{\R^n} ug^k \vd x \md \BM^k_s \right\vert^{\frac{p}{2}}\biggr],
\end{align}
where $C=C(\pc,T)$.
By the Burkholder--Davis--Gundy inequality (see \cite{revuz1999continuous}), 
the last term is dominated by
\begin{align}\label{newBDG}
& C\E\biggl[
\biggl( \int_0^\tau 
\sum_{k=1}^{\infty} \biggl\vert\int_{\R^n} ug^k 
\vd x\biggr\vert^2 \md t \biggr)^{\frac{p}{4}}\biggr] \nonumber\\
& \le C \E\biggl[
\biggl( \sup_{t\in[0,\tau]}\int_{\R^n} |u(x,t)|^2 
\vd x \biggr)^{\frac{p}{4}}
\biggl(\int_0^\tau\!\!\int_{\R^n} \|g(x,t)\|_{\ell^2}^2 
\vd x \md t\biggr)^{\frac{p}{4}}\biggr] \\
& \le \eps\, \E \sup_{t\in[0,\tau]}\|u(t)\|^p_{L^2_x}
+ C\eps^{-1}\|g\|_{L^p_\om L^2_t L^2_{x}(\Q_T)}^p. \nonumber
\end{align}
Taking the positive number $\eps$ sufficiently small and combining \eqref{newBDG} along with \eqref{eq301}, we thus obtain that
\begin{align*}
\E\biggl[
\biggl( \sup_{t\in[0,\tau]}\|u(t)\|^2_{L^2_x} + \int_0^\tau \|u(t)\|^2_{H^1_x} \vd t
\biggr)^{\!\frac{p}{2}} \biggr]
\le C \bigl(\|f\|^p_{L^p_\om L^2_t H^{-1}_{x}(\Q_T)} 
+ \|g\|^p_{L^p_\om L^2_t L^2_{x}(\Q_T)} \bigr),
\end{align*}
where the constant $C$ depends only on $\la$ and $T$.
Then \eqref{p22est0} follows
by applying the above estimate to the following sequence of stopping times
\[
\tau_k := \inf \biggl\{t\ge 0 : 
\sup_{s\in[0,t]} \|u(s)\|^2_{L^2_x} + \int_0^t \|u(s)\|^2_{H^1_x} \vd s
> k \biggr\} \wedge T,
\]
and sending $k$ to infinity.
For $m\geq1$, one can easily apply the induction argument to conclude the lemma. 
\end{proof}

\begin{proof}[Proof of Proposition \ref{local}]
Now we are ready to prove the first estimate in Proposition \ref{local}.
It suffices to consider the case $r=1$.
For general $r>0$, we can apply the obtained estimates for $r=1$ to the rescaled function
\[
v(x,t) := u(rx, r^2 t),
\quad
\forall\,(x,t)\in\R^n\x\R
\]
that solves the equation
\begin{equation}
\md v(x,t) = [\aij(r^2 t) v_\dij(x,t) + F(x,t)]\vd t 
+ [\sik(r^2 t) v_\di(x,t) + G^k(x,t)] \vd \be^k_t,
\end{equation}
with
\[
F(x,t) = r^2 f(rx, r^2 t),
\quad
G(x,t) = r g(rx, r^2 t),
\quad
\be^k_t = r^{-1} \BM^k_{r^2t},
\]
and obviously, $\be^k$ are mutually independent Wiener processes.

By induction, we shall only consider the case of $m=1$.
Let $\zeta \in C_0^\infty(\R^{n+1})$ be a nonnegative function 
such that $\zeta(x,t)= 1$ if $\pdist{(x,t)} \le \sqrt{\tht}$, where $\tht\in(0,1)$,
and $\zeta(x,t) = 0$ if $\pdist{(x,t)} \ge (1+\sqrt{\tht})/2$.
Then $v=\zeta u$ satisfies
\begin{equation}\label{mdf01}
\md v = (\aij \Dij v + \tilde{f} )\vd t + (\sik \Di v + \tilde{g}^k) \vd \BM^k_t,
\end{equation}
where
\[
\tilde{f} = \zeta f - 2 \aij (\zeta_\di u)_\dj + \aij\zeta_\dij u + \zeta_t u,
\quad
\tilde{g}^k = \zeta g^k - \sik \zeta_\di u.
\]
Applying Lemma \ref{p2m} to \eqref{mdf01} with $|\be|=0$, we have
\begin{align}
\label{3001}& \|u\|_{L^p_\om L^\infty_t L^2_{x}(Q_{\sqrt{\tht}})}
+ \|D u\|_{L^p_\om L^2_t L^2_{x}(Q_{\sqrt{\tht}})} \\
\nonumber & \quad \le C \bigl( \|u\|_{L^p_\om L^2_t L^2_{x}(Q_{1})}
+ \|f\|_{L^p_\om L^2_t L^2_{x}(Q_{1})}
+ \|g\|_{L^p_\om L^2_t L^2_{x}(Q_{1})} \bigr).
\end{align}

While by choosing another cut-off function $\zeta$ such that 
$\zeta(x,t)= 1$ if $\pdist{(x,t)} \le \tht$,
and $\zeta(x,t) = 0$ if $\pdist{(x,t)} \ge \sqrt{\tht}$,
and again applying Lemma \ref{p2m} with $|\be|=1$, we have
\begin{align}
\label{3002} & \|D u\|_{L^p_\om L^\infty_t L^2_{x}(Q_{\tht})}
+ \|D^2 u\|_{L^p_\om L^2_t L^2_{x}(Q_{\tht})} \\
\nonumber & \quad \le C \bigl( \|D u\|_{L^p_\om L^2_t L^2_{x}(Q_{\sqrt{\tht}})}
+ \|f\|_{L^p_\om L^2_t L^2_{x}(Q_{1})}
+ \|g\|_{L^p_\om L^2_t H^1_{x}(Q_{1})} \bigr).
\end{align}
Combining \eqref{3001} and \eqref{3002}, the first inequality in Proposition \ref{local} is proved.
\end{proof}

As an immediate application, we give an estimate for equation \eqref{model}
with the Dirichlet boundary conditions
\begin{equation}\label{dirich}
u(0,\cdot) = 0,\quad
u|_{\pd B_r} = 0.
\end{equation}
Denote $\Q_r = B_r \x (0,r^2)$.

\begin{proposition}\label{Lpdom}
Let $f$ and $g$ be in $L^\p_\om L^2_t H^m_{x}(\Q_{r})$ for all $m\in\Nat$.
Then the Dirichlet problem \eqref{model} and \eqref{dirich} 
has a unique weak solution $u\in L^2_\om L^2_t H^1_x (\Q_{r})$,
and for each $t\in(0,r^2)$, $u(t,\cdot)\in L^\p(\PS; C^{m}(B_\eps))$ for all $m\ge0$ and $\eps\in (0,r)$.
Moreover, there is a constant $C = C(n,\p)$ such that
\begin{align}\label{Lp}
\|u\|_{L^\p_\om L^2_t L^2_{x}(\Q_r)}
\le C \bigl( r^{2} \|f\|_{L^\p_\om L^2_t L^2_{x}(\Q_r)} 
+ r \|g\|_{L^\p_\om L^2_t L^2_{x}(\Q_r)} \bigr).
\end{align}
\end{proposition}

\begin{proof}
The existence, uniqueness and smoothness of the weak solution of 
the Dirichlet problem \eqref{model} and \eqref{dirich} 
follow from \cite[Theorem~2.1]{krylov1994aw}.
And the estimate \eqref{Lp} can be derived analogously to that of \eqref{p22est0}
by means of It\^o's formula and rescaling.
\end{proof}

\section{Interior H\"older estimates for the model equation}\label{sec4}

The aim of this section is to prove the interior H\"older estimates
for the model equation~\eqref{model} where $a$ and $\si$ are independent of $x$.
To be more general,
we assume that $f\in C^0_{x}(\R^n\x\R;\bL^\p)$ 
and $g\in C^{1}_{x}(\R^n\x\R;\bL^\p)$,
and $f(t,x)$ and $g_x(t,x)$ 
are Dini continuous with respect to $x$ uniformly in $t$, namely,
the modulus of continuity defined by
\[
\cm(r) = \sup_{t\in\R,\,|x-y| \le r} 
(\|f(t,x) - f(t,y)\|_{\bL^\p}+\|g_x(t,x) - g_x(t,y)\|_{\bL^\p})
\]
satisfies that
\[
\int_0^1 \frac{\cm(r)}{r} \vd r < \infty.
\]
Recall the notation $B_r$ and $Q_r$ defined in \eqref{BQ}.
The main estimate is the following
\begin{theorem}\label{Dini}
Let $u$ be a quasi-classical solution to \eqref{model}.
Under the above settings, there is a positive constant $C$, depending only on
$n,\pc$ and $\p$, such that for any $X,Y \in Q_{1/4}$,
\begin{equation}\label{Dini.est}
\|u_{xx}(X) - u_{xx}(Y)\|_{\bL^\p}
\le C \left[ \delta M_1
+ \int_0^{\delta} \frac{\cm(r)}{r}\vd r
+ \delta \int_{\delta}^1 \frac{\cm(r)}{r^2}\vd r
\right],
\end{equation}
where $\delta = \pdist{X-Y}$ and
$ M_1 = \norm{u}_{0;Q_{1}} + \norm{f}_{0;Q_{1}}
+ \norm{g}_{1;Q_{1}}$.
\end{theorem}

An immediate consequence is the following interior H\"older estimate for \eqref{model},
where we denote $\Q_{r,T} = B_r \x [0,T]$ for $r,T>0$.

\begin{corollary}
Let $u$ be a quasi-classical solution of \eqref{model}, and $\al\in(0,1)$.
Suppose that $u$, $f$ and $g$ vanish when $t\le 0$.
Then there is a positive constant $C$, depending only on
$n,\pc,\p$ and $\al$, such that
\begin{equation}\label{xholder}
\hnorm{u_{xx}}_{(\al,\al/2);\Q_{1/4,T}}
\le C \left[ \norm{u}_{0;\Q_{1,T}}
+ \frac{\norm{f}_{\al;\Q_{1,T}} + \norm{g}_{1+\al;\Q_{1,T}}}{\al(1-\al)}\right]
\end{equation}
for any $T>0$, provided the right-hand side is finite.
\end{corollary}

\begin{proof}
It follows from \eqref{Dini.est} that for any $X=(x,t)\in \R^{n+1}$,
\begin{equation}\label{pf401}
\hnorm{u_{xx}}_{(\al,\al/2);Q_{1/4}(X)}
\le C \left[ \norm{u}_{0;Q_{1}(X)}
+ \frac{\norm{f}_{\al;Q_{1}(X)} + \norm{g}_{1+\al;Q_{1}(X)}}{\al(1-\al)}\right].
\end{equation}
Then we fix $x=0$ and let $t$ run through $[0,T]$; 
keeping in mind that $u$ vanishes when $t\le 0$, 
and using the localization property of H\"older norms (see \cite[Lemma~4.1.1]{krylov1996lectures}), we have
\begin{align*}
\hnorm{u_{xx}}_{(\al,\al/2);\Q_{1/4,T}}
& \le C(n,\alpha) 
\sup_{0\le t \le T}  \Big( \hnorm{u_{xx}}_{(\al,\al/2);Q_{1/4}(0,t)}
+ \norm{u}_{0;Q_{1/4}(0,t)}\Big)\\
& \le C \sup_{0\le t \le T} \left[ \norm{u}_{0;Q_{1}(0,t)}
+ \frac{\norm{f}_{\al;Q_{1}(0,t)} + \norm{g}_{1+\al;Q_{1}(0,t)}}{\al(1-\al)}\right]\\
& \le C \left[ \norm{u}_{0;\Q_{1,T}}
+ \frac{\norm{f}_{\al;\Q_{1,T}} + \norm{g}_{1+\al;\Q_{1,T}}}{\al(1-\al)}\right]. 
\end{align*}
The proof is complete.
\end{proof}

\begin{proof}[Proof of Theorem~\ref{Dini}]

Letting $\vf:\R^n \to \R$ be a nonnegative and symmetric mollifier (see Appendix) and 
$\vf^\eps(x)=\eps^n\vf(x/\eps)$, we 
define $u^\eps = \vf^\eps * u$, $f^\eps = \vf^\eps * f$ and $g^\eps = \vf^\eps * g$.
Under the condition of Theorem~\ref{Dini}, it follows from Corollary~\ref{corA5} (see Appendix) that
\begin{gather*}
\norm{f^\eps - f}_{0;\R^n} + \norm{g^\eps - g}_{1;\R^n} \to 0,\\
\|D^2 u^\eps(X)-D^2 u(X)\|_{\bL^\p} \to 0
\quad \forall\,X\in\R^n \times \R,
\end{gather*}
as $\eps\to 0$.
Evidently, $f^\eps$ and $Dg^\eps$ are also Dini continuous and has the same modulus of continuity $\cm$ with $f$ and $Dg$.
On the other hand, from Fubini's theorem one can check that $u^\eps$ satisfies the model equation~\eqref{model} in the classical sense with free terms $f^\eps$ and $g^\eps$.
Therefore, it suffices to prove the theorem for the mollified functions,
and the general case is straightforward by passing the limits.

Based on the above analysis and the property of mollified functions (see Lemmas \ref{lemA1} and \ref{lemA2} and Remark \ref{remA1} in Appendix), we may assume that $f$ and $g$ satisfy the following additional condition:
\begin{itemize}
\item[({\bf A})] $f,g\in L^\p_\om L^2_t H^k_x (Q_R)\cap C^k_x(Q_R;\bL^\p)$ for all $k\in\Nat$ and $R>0$. 
\end{itemize}
\smallskip

From the definition of $\cm$, one can see that for any $x,y\in\R^n$ and $t\in\R$,
\begin{equation}\label{fgappr}
\begin{split}
& \|f(x,t)-f(y,t)\|_{\bL^\p} + \|g_x(x,t) - g_x(y,t)\|_{\bL^\p}
\le \cm(|x-y|), \\
& \|g(x,t)-g(y,t)-g_x(x,t) \cdot (x-y)\|_{\bL^\p}
\le |x-y|\cm(|x-y|).
\end{split}
\end{equation}
With $\rho = 1/2$, we denote
\[
Q^\m = Q_{\rho^\m} = Q_{\rho^\m}(0,0),
\quad
\m = 0,1,2,\cdots.
\]

Let us introduce the following boundary problems:
\begin{align*}
\md u^\m & = [\aij u^\m_\dij + f(0,t) ]\vd t 
+ [\sik u^\m_\di + g^k(0,t) + g^k_x(0,t) \cdot x] \vd \BM^k_t
\quad \text{in } Q^\m, \\
u^\m & = u \quad \text{on } \pd_{\rm p} Q^\m,
\end{align*}
where $\pd_{\rm p} Q^\m$ denotes the parabolic boundary of the cylinder $Q^\m$ for $\m=0,1,2,\dots$.
Applying Proposition~\ref{Lpdom} to the equation of $u^\m - u $, we can obtain the the solvability and interior regularity of each $u^\m$.

\vspace{2pt}
Now, we {\em claim} that there is a constant $C = C(n,\pc,\p)$ such that 
\begin{equation}\label{diff}
\norm{D^m(u^\m - u^{\m+1})}_{0;Q^{\m+2}} 
\le C \rho^{(2-m)\m-m} \cm(\rho^\m), 
\quad 
m=1,2,\dots.
\end{equation}
To see this, we apply the second estimate in Proposition \ref{local} (with $f$ and $g$ vanishing) to $u^\m - u^{\m+1}$ with $|\beta|=m, r=\rho^{\m+1}, \theta=1/2$ and $p=\gamma$ to get
\[
\norm{D^m (u^\m - u^{\m+1})}_{0;Q^{\m+2}}
\le C \rho^{-m \m - m} 
\biggl{\|} \fint_{Q^{\m+1}} (u^\m - u^{\m+1})^2 \vd X \biggr{\|}_{\bL^{\p/2}}^{1/2}
=: I_{\m,m}.
\]
Here and in what follows, we denote $\fint_Q  = {1 \over |Q|}\int_Q$ with $|Q|$ being the Lebesgue measure of the set $Q\subset \R^{n+1}$.

On the other hand, it follows from Proposition~\ref{Lpdom} that
\[
J_\m := \biggl{\|} \fint_{Q^{\m}} (u^\m - u)^2 \vd X \biggr{\|}_{\bL^{\p/2}}^{1/2}
\le C \rho^{2\m} \cm(\rho^\m).
\]
Combining the above we obtain
\[
I_{\m,m} \le C \rho^{-m\m - m} (J_\m + J_{\m+1}) \le C \rho^{(2-m)\m - m}\cm(\rho^\m)
\]
and thus the claim \eqref{diff}.

\vspace{2pt}
The estimate~\eqref{diff} with $m=2$ gives (recalling $\rho=1/2$)
\[
\sum_{\m\ge 1}\norm{(u^\m - u^{\m+1})_{xx}}_{0;Q^{\m+2}} 
\le C \rho^{-2}\sum_{\m\ge 1} \cm(\rho^\m) 
\le 4C \int_0^1 \frac{\cm(r)}{r} \vd r
< \infty,
\]
which implies that $u^\m_{xx}(0)$ converges in $\bL^\p$ as $\m\to\infty$,
(here $0\in\R^{n+1}$).
We shall prove that the limit is $u_{xx}(0)$.
Since $\p\ge 2$, it suffices to show that
\begin{equation}\label{conv}
\lim_{\m\to\infty}\|u^\m_{xx}(0) - u_{xx}(0)\|_{\bL^2}
=0.
\end{equation}
Applying the second estimate in Proposition~\ref{local} to $u^\m - u$ with $m=n+2, |\beta|=2, r=\rho^\m, \theta=1/2$ and $p=2$, we have
\begin{gather*}
\sup_{Q^{\m+1}}\|u^\m_{xx} - u_{xx}\|_{\bL^2}^2
\le C \rho^{-4\m} \,\E \fint_{Q^\m} |u^\m - u|^2 \vd X \\
\quad + C \,\E\fint_{Q^\m}\bigl( |f(x,t) - f(0,t)|^2 
+ \|g(x,t) - g(0,t)\|^2 + \|g_x(x,t) - g_x(0,t)\|^2 \bigr)\vd X \\
\quad + C \sum_{k=1}^{n+1} \rho^{2\m k} 
\,\E\fint_{Q^\m}\bigl( |D^{k} f|^2 + \|D^{k+1} g\|^2 \bigr)\vd X.
\end{gather*}
According to the additional condition ({\bf A}) on $f$ and $g$, it is clear that
the last two terms on the right-hand side tend to zero as $\m \to \infty$.
Moreover, from Proposition~\ref{Lpdom} and \eqref{fgappr} we have
\begin{align*}
& \rho^{-4\m} \,\E \fint_{Q^\m} |u^\m - u|^2 \vd X \\
& \le C \,\E\fint_{Q^\m}\bigl( |f(x,t) - f(0,t)|^2 
+ \rho^{-2\m}|g(x,t) - g(0,t) - g_x(0,t)\cdot x|_{l^2}^2 \bigr)\vd X\\
& \le C \cm(\rho^\m)^2
\to 0,
\quad
\text{as }
\m \to \infty.
\end{align*}
Therefore, \eqref{conv} is proved and $u^\m_{xx}(0)$ converges strongly to $u_{xx}(0)$ in $\bL^\p$.
Moreover, by means of \eqref{diff}, we have
\begin{equation}\label{conv.rate}
\|u^\m_{xx}(0) - u_{xx}(0)\|_{\bL^\p}
\le \sum_{j\ge \m}\norm{(u^j - u^{j+1})_{xx}}_{0;Q^{j+2}} 
\le C \int_0^{\rho^\m} \! \frac{\cm(r)}{r} \vd r,
\end{equation}
where $C = C(n,\pc,\p)$.

\vspace{2pt}
Next we estimate the oscillation of $u^\m_{xx}$.
Starting from $\m=0$, $u_{xx}^0$ satisfies the following homogeneous equation:
\begin{align}\label{u0-eq}
\md u^0_{xx} & = \aij \Dij u^0_{xx} \vd t + \sik \Di u^0_{xx} \vd \BM^k_t
\quad \text{in } Q_{3/4}.
\end{align}
Using the second estimate in Proposition~\ref{local} (with $f$ and $g$ vanishing) to $u^0_{xx}$, we have
\begin{align*}
& \norm{D_x u^0_{xx}}_{0;Q_{1/4}} + \norm{D_x^2 u^0_{xx}}_{0;Q_{1/4}}
\le C \|u^0_{xx}\|_{\bL^\p L^2_t L^2_x(Q_{1/2})} \\
& \quad \le C (\|u^0_{xx} - u_{xx}\|_{\bL^\p L^2_t L^2_x(Q_{1/2})} 
+ \|u_{xx}\|_{\bL^\p L^2_t L^2_x(Q_{1/2})}).
\end{align*}
Then we apply the first estimate in Proposition~\ref{local} to $u$ to get
\[
\|u_{xx}\|_{\bL^\p L^2_t L^2_x(Q_{1/2})}
\le C (\|u\|_{\bL^\p L^2_t L^2_x(Q_{1})} 
+ \|f\|_{\bL^\p L^2_t L^2_x(Q_{1})} 
+ \|g\|_{\bL^\p L^2_t H^1_x(Q_{1})}),
\]
and to $u^0 - u$ along with Proposition~\ref{Lpdom},
\begin{align*}
\|u^0_{xx} - u_{xx}\|_{\bL^\p L^2_t L^2_x(Q_{1/2})} 
& \le C (\|u^0 - u\|_{\bL^\p L^2_t L^2_x(Q_{1})} 
+ \|f\|_{\bL^\p L^2_t L^2_x(Q_{1})} 
+ \|g\|_{\bL^\p L^2_t H^1_x(Q_{1})})\\
& \le C (\|f\|_{\bL^\p L^2_t L^2_x(Q_{1})} 
+ \|g\|_{\bL^\p L^2_t H^1_x(Q_{1})}).
\end{align*}
Therefore,
\begin{align*}
& \norm{D_x u^0_{xx}}_{0;Q_{1/4}} + \norm{D_x^2 u^0_{xx}}_{0;Q_{1/4}}\\
& \le C (\|u\|_{\bL^\p L^2_t L^2_x(Q_{1})} 
+ \|f\|_{\bL^\p L^2_t L^2_x(Q_{1})} 
+ \|g\|_{\bL^\p L^2_t H^1_x(Q_{1})}) 
\le C M_1.
\end{align*}
Hence, for $-1/16 < s\le t\le 0$ and $x\in B_{1/4}$, 
\begin{align*}
& \|u^0_{xx}(x,t) - u^0_{xx}(x,s)\|_{\bL^\p}
= \bigg{\|} \int_s^t \aij \Dij u^0_{xx} \vd \tau 
+ \int_s^t  \sik \Di u^0_{xx} \vd \BM^k_\tau \bigg{\|}_{\bL^\p}\\
& \quad
\le C \sqrt{t-s} (\norm{D u^0_{xx}}_{0;Q_{1/4}} + \norm{D^2 u^0_{xx}}_{0;Q_{1/4}})
\le C \sqrt{t-s} M_1,
\end{align*}
where $C = C(n,\pc,\p)$.
Thus, we obtain
\begin{equation}\label{u0est}
\|u_{xx}^0(X) - u_{xx}^0(Y)\|_{\bL^\p}
\le C M_1 \pdist{X-Y},
\quad\forall\,X,Y\in Q_{1/4}.
\end{equation}

\vspace{2pt}
To deal with $u^\m_{xx}$ with $\m \ge 1$, we denote
\[
h^\iota = u^\iota - u^{\iota-1},
\quad\text{for } \iota = 1,2,\dots,\m.
\]
Then $h^\iota$ satisfies 
\begin{align}\label{h0-eq}
\md h^\iota & = \aij h^\iota_\dij \vd t + \sik h^\iota_\di \vd \BM^k_t
\quad \text{in } Q^{\iota}.
\end{align}
By \eqref{diff} we have
\begin{align*}
\rho^{-\iota} \norm{D^3h^\iota}_{0;Q^{\iota+1}}
+ \norm{D^4 h^\iota}_{0;Q^{\iota+1}}
\le C \rho^{-2\iota}\cm(\rho^{\iota-1}).
\end{align*}
Hence, for $-\rho^{2(\m+1)} \le t \le 0$ and $|x| \le \rho^{\m+1}$,
\[
\|h^\iota_{xx}(x,0) - h^\iota_{xx}(0,0)\|_{\bL^\p}
\le  C \rho^{\m-\iota}\cm(\rho^{\iota-1})
\]
and
\begin{align*}
& \|h^\iota_{xx}(x,t) - h^\iota_{xx}(x,0)\|_{\bL^\p}
= \bigg{\|} \int_t^0 \aij \Dij h^\iota_{xx} \vd \tau 
+ \int_t^0  \sik \Di h^\iota_{xx} \vd \BM^k_\tau \bigg{\|}_{\bL^\p}\\
& \quad
\le C \rho^{2\m} \norm{D^4 h^\iota}_{0;Q^{\iota+1}}
+ C \rho^{\m} \,\norm{D^3 h^\iota}_{0;Q^{\iota+1}}
\le C \rho^{\m-\iota}\cm(\rho^{\iota-1}).
\end{align*}
Let $Y=(y,s)\in Q_{1/4}$, and $\tilde{\m}\in\Nat$ such that  
\[\delta := \pdist{Y} \in [\rho^{\tilde{\m}+2}, \rho^{\tilde{\m}+1}).\]
Combining the last two estimates and \eqref{u0est},
we can obtain 
\begin{equation*}\label{ukest}
\begin{split}
\|u^{\tilde{\m}}_{xx}(Y) - u^{\tilde{\m}}_{xx}(0)\|_{\bL^\p}
& \le \|u^{{\tilde{\m}}-1}_{xx}(Y) - u^{{\tilde{\m}}-1}_{xx}(0)\|_{\bL^\p}
+ \|h^{\tilde{\m}}_{xx}(Y) - h^{\tilde{\m}}_{xx}(0)\|_{\bL^\p} \\
& \le \|u^{0}_{xx}(Y) - u^{0}_{xx}(0)\|_{\bL^\p}
+ \sum_{\iota=1}^{\tilde{\m}} \|h^\iota_{xx}(Y) - h^\iota_{xx}(0)\|_{\bL^\p} \\
& \le CM_1 \rho^{{\tilde{\m}}+1}
+ C \sum_{\iota=1}^{\tilde{\m}} \rho^{{\tilde{\m}}-\iota}\cm(\rho^{\iota-1}) \\
& \le CM_1 \rho^{{\tilde{\m}}+2} + C \rho^{{\tilde{\m}}+2} 
\int_{\rho^{{\tilde{\m}}}}^1  \frac{\cm(r)}{r^2} \vd r \\
& \le C \delta M_1 + C  \delta
\int_{\delta}^1  \frac{\cm(r)}{r^2} \vd r.
\end{split}
\end{equation*}

\vspace{2pt}
By virtue of \eqref{conv.rate} we have the following decomposition 
\begin{align}
& \|u_{xx}(Y) - u_{xx}(0)\|_{\bL^\p}  \label{pf402}\\
 \le &\ \|u^{\tilde{\m}}_{xx} (Y) - u^{\tilde{\m}}_{xx} (0)\|_{\bL^\p}
+ \|u^{\tilde{\m}}_{xx} (0) - u_{xx} (0)\|_{\bL^\p}
+ \|u^{\tilde{\m}}_{xx} (Y) - u_{xx}(Y)\|_{\bL^\p} \nonumber\\
 \le &\ C \left[ \delta M_1 + \int_0^{4\delta} \frac{\cm(r)}{r}\vd r
+ \delta \int_\delta^{1} \frac{\cm(r)}{r^2}\vd r
\right]
+ \|u_{xx}^{\tilde{\m}} (Y) - u_{xx}(Y)\|_{\bL^\p}. \nonumber
\end{align}
It remains to estimate the last term in the above inequality.
To this end, we consider the sequence of equations
\begin{align*}
\md u^{Y,{{\m}}} & = [\aij u^{Y,{{\m}}}_\dij + f(y,t) ]\vd t 
+ [\sik u^{Y,{{\m}}}_\di + g^k(y,t) + g^k_x(y,t) \cdot x] \vd \BM^k_t
\quad \text{in } Q^{{\m}}(Y), \\
u^{Y,{{\m}}} & = u \quad \text{on } \pd_{\rm p} Q^{{\m}}(Y)
\quad \text{with } 
\m = 0,1,\dots,\tilde{\m}-1,\tilde{\m}+2,\dots;
\end{align*}
the equations associated with $\tilde{\m}$ and $\tilde{\m}+1$ are replaced by the following single equation
\begin{align*}
\md u^{Y,{\tilde{\m}}} & = [\aij u^{Y,{\tilde{\m}}}_\dij + f(y,t) ]\vd t 
+ [\sik u^{Y,{\tilde{\m}}}_\di + g^k(y,t) + g^k_x(y,t) \cdot x] \vd \BM^k_t
\quad \text{in } Q^{\tilde{\m}}(0), \\
u^{Y,{\tilde{\m}}} & = u \quad \text{on } \pd_{\rm p} Q^{\tilde{\m}}(0).
\end{align*}
As $\pdist{Y} \in [\rho^{\tilde{\m}+2}, \rho^{\tilde{\m}+1})$, it is easily seen that $Q^{\tilde{\m}+2}(Y)\subset Q^{\tilde{\m}}(0)\subset Q^{\tilde{\m}-1}(Y)$.
So analogously to proving \eqref{conv.rate} but only with mirror changes, one can derive 
\begin{equation}\label{conv.rateY}
\|u^{Y,{\tilde{\m}}}_{xx}(Y) - u_{xx}(Y)\|_{\bL^\p}
\le C \int_0^{\rho^{\tilde{\m}}} \! \frac{\cm(r)}{r} \vd r,
\end{equation}
where $C = C(n,\pc,\p)$.
On the other hand, applying Proposition~\ref{local} to 
the equation satisfied by $u^{Y,{\tilde{\m}}} - u^{\tilde{\m}}$, and using~\eqref{fgappr}, we have
\begin{align*}
\|u^{Y,{\tilde{\m}}}_{xx}(Y) - u^{\tilde{\m}}_{xx}(Y)\|_{\bL^\p} 
\le C \bigl(\|u^{Y,{\tilde{\m}}} - u^{\tilde{\m}}\|_{\bL^\p L^2_t L^2_x(Q^{\tilde{\m}})} + \cm(\delta)\bigr),
\end{align*}
while by Proposition~\ref{Lpdom},
\begin{align*}
\|u^{Y,{\tilde{\m}}} - u^{\tilde{\m}}\|_{\bL^\p L^2_t L^2_x(Q_{1})}
\le C \cm(\delta).
\end{align*}
Thus
\[
\|u_{xx}^{\tilde{\m}} (Y) - u_{xx}(Y)\|_{\bL^\p}
\le C \cm(\delta) + C \int_0^{4\delta} \! \frac{\cm(r)}{r} \vd r.
\]
Substituting the above estimate into \eqref{pf402}, we then complete the proof.
\end{proof}

\begin{remark}
Consider the model equation of divergence-form
\begin{equation}\label{model-div}
\md u = (\aij u_\dj + f^i )_\di\vd t + (\sik u_\di + g^k) \vd \BM^k_t.
\end{equation}
With the help of the following approximation sequence
\begin{align*}
\md u^\m & = \aij u^\m_\dij \vd t 
+ [\sik u^\m_\di + g^k(0,t)] \vd \BM^k_t
\quad \text{in } Q^\m, \\
u^\m & = u \quad \text{on } \pd_{\rm p} Q^\m,
\end{align*}
we can similarly obtain an interesting estimate
\begin{gather*}
\hnorm{u}_{(1+\al,\al/2);Q_{1/4}}
\le C \bigl(
\norm{u}_{0;Q_1}
+ \sum\nolimits_i \norm{f^i}_{\al;Q_1}
+ \norm{g}_{\al;Q_1}
\bigr),
\end{gather*}
provided the right-hand side is finite.
The result on the model equation \eqref{model-div} can help us establish a $C^{1+\al}$ estimate 
for more general equations, which will be discussed in a separate work.
\end{remark}

\section{Global H\"older estimates for general equations}\label{sec5}

This section is devoted to the proof of Theorem \ref{schauder1}.
First we state two technical lemmas.

\begin{lemma}\label{iterat}
Let $\vf$ be a bounded nonnegative function defined
on $[0,T]$ satisfying
\begin{equation}\label{lem501}
\vf(t) \le \tht \vf(s) + \sum_{i=1}^m A_i (s-t)^{-\delta_i},
\quad
\forall\,0\le t < s\le T,
\end{equation}
for some nonnegative constants $\tht,\delta_i$ and $A_i$ ($i=1,\dots,m$), where $\tht <1$.
Then
\[
\vf(0) \le C 
\sum_{i=1}^m A_i T^{-\delta_i},
\]
where $C$ depends only on $\delta_1,\dots,\de_m$ and $\tht$.
\end{lemma}

\begin{proof}
We may suppose $T=1$, otherwise let $\tilde{\vf}(t) = \vf(Tt)$.
Then \eqref{lem501} implies 
\[
\vf(t) \le \tht \vf(s) + A (s-t)^{-\de},
\quad
\forall\,0\le t < s\le 1,\]
where $\de:= \max_{1\le i\le m} \de_i$ and $A := A_1+\cdots+A_m$.
It suffices to consider $\de > 0$.
Take $\tau\in (0,1)$ such that $\eps:=\tht \tau^{-\de} < 1$,
and set $t_0 = 0$, $t_{j+1} = t_j + (1-\tau)\tau^j$ for $j=0,1,\dots$.
Then
\begin{align*}
\tht^{j}\vf(t_{j})
\le \tht^{j+1} \vf(t_{j+1}) + A (1-\tau)^{-\de}(\tht\tau^{-\de})^j
= \tht^{j+1} \vf(t_{j+1}) + \eps^j A (1-\tau)^{-\de}.
\end{align*}
By iteration, we gain
\begin{align*}
\vf(0)
& \le \tht^k \vf(t_k) + (1+\eps+\cdots+\eps^{k-1})A (1-\tau)^{-\de}\\
& \le \tht^k \vf(t_k) + (1-\eps)^{-1}A (1-\tau)^{-\de}.
\end{align*}
By letting $k\to\infty$, we conclude the proof.
\end{proof}

\begin{lemma}\label{interp}
Let $B_R = \{x\in\R^n:|x|<R\}$ with $R>0$, $E$ be a Banach space, $p \ge 1$, and $0\le s < r$. 
There exists a positive constant $C$, depending only on $n$ and $p$,
such that
\begin{gather}\label{itpl.est}
[u]^E_{s;B_R} \le C\eps^{r-s} [u]^E_{r;B_R} + C \eps^{-s-n/p} \|u\|_{L^p(B_R;E)}
\end{gather}
for any $u\in C^r(B_R;E)$ and $\eps \in (0,R)$.
\end{lemma}

\begin{proof}
Let us consider $R=1$ first.
Following the proof of classical interpolation inequalities for H\"older norms 
(see \cite[Lemma~6.35]{gilbarg2001elliptic} or \cite[Theorem~3.2.1]{krylov1996lectures}), 
one can derive
\begin{align}\label{cls.itpl}
[u]^E_{s;B_1} 
\le C\eps^{r-s} [u]^E_{r;B_1} + C \eps^{-s} |u|^E_{0;B_1}.
\end{align}
Consider the case of $r\le 1$ first.
For arbitrary $x\in B_1$, we select a ball $B_\eps(y) = \{z:|y-z|<\eps\} \subset B_1$ such that $x\in B_\eps(y)$.
Then we compute
\begin{align*}
\|u(x)\|_E 
& = \fint_{B_\eps(y)} \| u(x) \|_E \vd z \le \fint_{B_\eps(y)} \| u(x) -  u(z)\|_E\vd z 
+ \fint_{B_\eps(y)} \| u(z) \|_E \vd z \\
& \le C \eps^{r} [u]^E_{r;B_1} + C \eps^{-n/p} \|u\|_{L^p(B_1;E)},
\end{align*}
which yields \eqref{itpl.est} with $s=0$ and $r\le 1$.
For the case of $r>1$, we have
\begin{align*}
|u|^E_{0;B_1} 
& \le C\eps [u]^E_{1;B_1} + C \eps^{-n/p} \|u\|_{L^p(B_1;E)} \\
& \le C \eps\eps_1^{r-1} [u]^E_{r;B_1} +  C \eps\eps_1^{-1} |u|^E_{0;B_1}
+ C \eps^{-n/p} \|u\|_{L^p(B_1;E)}.
\end{align*}
Choosing $\eps_1 = 2C\eps$, we get \eqref{itpl.est} with $s=0$ and $r > 1$.
Finally, for the general case, we derive
\begin{align}\label{eq:B1}
|u|^E_{s;B_1} 
\le C\eps^{r-s} [u]^E_{r;B_1} +  C \eps^{-s} |u|^E_{0;B_1} 
\le C\eps^{r-s} [u]^E_{r;B_1} + C \eps^{-s-n/p} \|u\|_{L^p(B_1;E)}.
\end{align}
The case of $R=1$ is proved.

Now we turn to the general $R>0$.
With $v(x):=u(Rx)$ we have that $[v]^E_{r;B_1} = R^r [u]^E_{r;B_R}$
and $\|v\|_{L^p(B_1;E)} = R^{-n/p}\|u\|_{L^p(B_R;E)}$.
Applying \eqref{eq:B1} to $v$
we can obtain that
\begin{gather*}\label{itpl.est}
[u]^E_{s;B_R} \le C(R\eps)^{r-s} [u]^E_{r;B_R} + C (R\eps)^{-s-n/p} \|u\|_{L^p(B_R;E)}
\end{gather*}
for any $u\in C^r(B_R;E)$ and $\eps \in (0,1)$.
The proof is complete.
\end{proof}

We are now ready to prove Theorem~\ref{schauder1}.

\begin{proof}[Proof of Theorem~\ref{schauder1}]
Let $\rho/2\le r<R\le \rho$ with $\rho\in (0,1/4)$ to be specified.
Take a nonnegative function $\zeta\in C_0^\infty(\R^n)$ 
such that $\zeta(x)=1$ when $|x|\le r$; $\zeta(x)=0$ when $|x|>R$, and for $\delta \ge 0$,
\[
[\zeta]_{\delta;\R^n} \le C(R-r)^{-\delta}.
\]
Set $v=\zeta {u}$,
and
\[
\aij_{\circ}(t) = \aij(0,t),
\quad
\sik_{\circ}(t)=\sik(0,t).
\]
Then $v$ satisfies
\begin{equation}\label{cutoff}
\md v = (\aij_{\circ} v_\dij + \tilde{f})\vd t 
+ (\sik_\circ v_\di + \tilde{g}^k) \vd {\BM}^k_t,
\end{equation}
where
\begin{align*}
\tilde{f} = \,&  (\aij - \aij_{\circ})\zeta {u}_\dij
+ (\bi\zeta - 2\aij \zeta_j) {u}_\di + (c\zeta - \aij \zeta_\dij - \bi \zeta_i) {u}
+ \zeta {f},\\
\tilde{g}^k = \, &  (\sik - \sik_{\circ})\zeta {u}_\di
+ (\nu^k \zeta - \sik \zeta_i) {u} + \zeta {g}^k.
\end{align*}
For a positive number $\tau$, we set $\Q_{R,\tau} = B_R \x (0,\tau)$ and define
\[
{M}^\tau_{x,r}(u) 
= \sup_{0\le t \le \tau}
\biggl(\fint_{B_r(x)} \!\E\,|u(t,y)|^\p \vd y \biggr)^{1/\p},
\quad
{M}^\tau_{r}(u) = \sup_{x\in\R^n}{M}^\tau_{x,r}(u).
\]
Then by Lemma~\ref{interp},
\begin{align*}
\norm{\tilde{f}}_{\al;\Q_{R,\tau}}
& \le (\eps + K R^\al) \hnorm{{u}}_{2+\al;\Q_{R,\tau}}
+ C (R-r)^{-2-\al-n/\p} {M}^\tau_{0,R}(u) \\
& \quad + \hnorm{{f}}_{\al;\Q_{R,\tau}}
+ {C (R-r)^{-\al}} \norm{{f}}_{0;\Q_{R,\tau}}, \\
\norm{\tilde{g}}_{1+\al;\Q_{R,\tau}}
& \le (\eps + K R^\al) \hnorm{u}_{2+\al;\Q_{R,\tau}}
+ C (R-r)^{-2-\al-n/\p} {M}^\tau_{0,R}(u) \\
& \quad + \hnorm{{g}}_{1+\al;\Q_{R,\tau}}
+ {C (R-r)^{-1-\al}} \norm{{g}}_{0;\Q_{R,\tau}}.
\end{align*}
where $C=C(n,K,\p,\eps,\rho)$. 
Take $\rho,\eps>0$ so small that $\eps + K \rho^\al \le 1/4$,
then by virtue of Corollary~\ref{xholder},
we obtain that for any $\rho/2\le r < R \le \rho$,
\begin{align*}
\hnorm{{u}}_{(2+\al,\al/2);\Q_{r,\tau}}
& \le \frac12 \hnorm{{u}}_{2+\al;\Q_{R,\tau}} 
+ {C (R-r)^{-2-\al-n/\p}} {M}^\tau_{0,R}(u)
+ \hnorm{{f}}_{\al;\Q_{R,\tau}} \\
& \quad  + \hnorm{{g}}_{1+\al;\Q_{R,\tau}} 
+ {C (R-r)^{-\al}} \norm{{f}}_{0;\Q_{R,\tau}}
+ {C (R-r)^{-1-\al}} \norm{{g}}_{0;\Q_{R,\tau}}.
\end{align*}

By Lemma~\ref{iterat}, we gain
\begin{align}\label{pf404}
\hnorm{{u}}_{(2+\al,\al/2);\Q_{\rho/2,\tau}}
\le C \Bigl(  {M}^\tau_{0,\rho}(u)
+ \norm{{f}}_{\al;\Q_{\rho,\tau}} 
+ \norm{{g}}_{1+\al;\Q_{\rho,\tau}} \Bigr).
\end{align}
We can move the centre of domain to any point $x \in \R^n$, thus 
\begin{align*}
\sup_{x\in\R^n}\hnorm{{u}}_{(2+\al,\al/2);\Q_{\rho/2,\tau}(x)}
\le C \Bigl( {M}^\tau_{\rho}(u)
+ \norm{{f}}_{\al;\Q_{\tau}} 
+ \norm{{g}}_{1+\al;\Q_{\tau}} \Bigr),
\end{align*}
which along with the localization property of H\"older norms (cf. \cite[Lemma~4.1.1]{krylov1996lectures}) and Lemma~\ref{interp}, we have
\begin{align}\label{pf403}
\hnorm{{u}}_{(2+\al,\al/2);\Q_{\tau}}
& \le C \sup_{x\in\R^n} \Big( 
\hnorm{{u}}_{(2+\al,\al/2);\Q_{\rho/2,\tau}(x)} 
+ \norm{{u}}_{0;\Q_{\rho/2,\tau}(x)}
\Big) \\
& \le C \Big( \sup_{x\in\R^n}
\hnorm{{u}}_{(2+\al,\al/2);\Q_{\rho/2,\tau}(x)} + {M}^\tau_{\rho/2}(u)
\Big) \nonumber \\
& \le C
\Bigl( {M}^\tau_{\rho}(u)
+ \norm{{f}}_{\al;\Q_{\tau}} 
+ \norm{{g}}_{1+\al;\Q_{\tau}} \Bigr), \nonumber
\end{align}
where $C = C(n,\pc,\p,\al)$ is a generic constant.

To estimate ${M}_\rho^{\tau}(u)$,
we apply It\^o's formula to compute
\begin{align*}
\md |{u}|^\p & = \p |{u}|^{\p-2} {u} (\aij {u}_\dij 
+ \bi {u}_\di + c {u} + {f})\vd t \\
& \quad + \frac{\p(\p-1)}{2} |{u}|^{\p-2} |\sik {u}_\di + \nu^k {u} + {g}^k|^2 \vd t
+ \md m_t,
\end{align*}
where $m_t$ is a martingale.
Integrating in $\Q_{\rho,\tau}\x\PS$,
and using the H\"older and Sobolev--Gagliargo--Nirenberg inequalities, we get
\begin{align*}
\E\int_{B_{\rho}}\! |u(x,\tau)|^\p \vd x \le 
C_1 \E \int_{\Q_{\rho,\tau}}\! \Bigl( |{u}_{xx}|^\p
+ |{u}|^\p + |{f}|^\p + |{g}|^\p \Bigr) \vd x \md t.
\end{align*}
Thus,
\begin{align*}
{M}^\tau_{0,\rho}(u) \le
C_1 \tau (\norm{u}_{2;\Q_{\rho,\tau}} + \norm{{f}}_{0;\Q_{\tau}} 
+ \norm{{g}}_{0;\Q_{\tau}}),
\end{align*}
where $C_1 = C_1(n,\pc,\p)$.
Letting $\tau = (2CC_1)^{-1}$,
the above inequality along with \eqref{pf403} yields 
\begin{align}\label{tau}
\norm{{u}}_{(2+\al,\al/2);\Q_{\tau}}
\le C_0 \bigl( \norm{{f}}_{\al;\Q_{\tau}} 
+ \norm{{g}}_{1+\al;\Q_{\tau}} \bigr),
\end{align}
where $C_0 = C_0(n,\pc,\p,\al)$.

Let us conclude the proof by induction.
For $S>0$, assume that there is a constant $C_S$ such that 
\begin{align}\label{timeS}
\norm{{u}}_{(2+\al,\al/2);\Q_{S}}
\le C_S \bigl( \norm{{f}}_{\al;\Q_{S}} 
+ \norm{{g}}_{1+\al;\Q_{S}} \bigr).
\end{align}
With $u^S := u(\cdot,S)$, it is easily seen that $v := (u - u^S)$ satisfies 
\begin{equation*}
\begin{split}
& \md v = [\aij v_\dij + \bi v_\di + c v 
+ ({f} + \aij u^S_\dij + \bi u^S_\di + cu^S)]\vd t \\
& \qquad + [\sik v_\di + \nu^k v 
+ ({g}^k + \sik u^S_\di + \nu^k u^S)] \vd \BM^k_t,
\quad \text{on } \R^n\x(S,\infty),\\
& v(S,x) = 0, \qquad x\in \R^n.
\end{split}
\end{equation*}
Applying \eqref{tau} to this equation and with \eqref{timeS} in mind, 
we have
\begin{align*}
\norm{u}_{(2+\al,\al/2);\Q_{S+\tau}}
& \le \norm{v}_{(2+\al,\al/2);\Q_{S+\tau}} + \norm{u}_{(2+\al,\al/2);\Q_{S}} \\
& \le C_0 \bigl( \norm{{f}}_{\al;\Q_{S+\tau}} 
+ \norm{{g}}_{1+\al;\Q_{S+\tau}} \bigr) +  C_0 N \norm{u}_{(2+\al,\al/2);\Q_{S}} \\
& \le C_0 (1 + N C_S) \bigl( \norm{{f}}_{\al;\Q_{S+\tau}} 
+ \norm{{g}}_{1+\al;\Q_{S+\tau}} \bigr), 
\end{align*}
where $N$ is a constant depending only on $n,\pc,\p$ and $K$.
Hence,
\[
C_{S+\tau} \le C_0 (1 + N C_S).
\]
As $\tau$ is fixed, by iteration we have
$C_S \le C\me^{CS}$, where $C = C(n,\pc,\p,\al,K)$.
This concludes the proof of estimate \eqref{schauder.est} and thus Theorem~\ref{schauder1}.
\end{proof}

\appendix

\setcounter{theorem}{0}
\renewcommand{\thetheorem}{A.\arabic{theorem}}
\setcounter{remark}{0}
\renewcommand{\theremark}{A.\arabic{remark}}
\setcounter{equation}{0}
\renewcommand{\theequation}{A.\arabic{equation}}

\section*{Appendix}

In this section we prove some properties and approximation of $\bL^\p$-valued continuous and differentiable functions.
Let $\Dom$ be a simply connected domain in $\R^n$.
Denote $C(\Dom;\bL^\p)$ the set of all $\bL^\p$-valued strongly continuous functions defined on $\Dom$ such that $\sup_{x\in\Dom}\E[|u(x)|^\p]<\infty$,
and $C^m(\Dom;\bL^\p)$ the set of all $C(\Dom;\bL^\p)$ functions whose strong derivatives up to order $m$ all exit and belong to $C(\Dom;\bL^\p)$, where $m$ is a nonnegative integer. 
In view of a known result (see \cite[Proposition~3.6]{da1992stochastic}), every function in $C(\Dom;\bL^\p)$ has a modification jointly measurable with respect to $x\in\Dom$ and $\om\in\PS$; we will always choose this modification.

In what follows, we denote $D u$ to the strong derivatives of an $\bL^\p$-valued differentiable function $u$, and $\partial u$ to be the classical derivatives if exist.

\begin{lemma}\label{lemA1}
If $u\in L^{\p}(\PS;\Ct^m(\Dom))$, then $u\in\Ct^m(\Dom;\bL^\p)$, and
$D^\beta u = \partial^\beta u$ for any multi-index $\be$ with $|\be|\le m$.
\end{lemma}

\begin{proof}
It follows from the dominated convergence theorem that
$L^{\p}(\PS;\Ct(\Dom)) \subset \Ct(\Dom;\bL^\p)$.
For $u\in L^{\p}(\PS;\Ct^1(\Dom))$,
we know that $\partial u \in \Ct(\Dom;\bL^\p)$, and by Jensen's inequality and Fubini's theorem,
\begin{align*}
& \E | r^{-1}[u(x+re_i) - u(x)] - \partial_i u(x) |^\p 
= \E \left| \int_0^1 \big[\partial_i u (x+sre_i) - \partial_i u(x) \big] \vd s \right|^\p \\
& \le \int_0^1 \E \big| \partial_i u (x+sre_i) - \partial_i u(x) \big|^\p \vd s 
 \to 0
\quad
\text{as }
r \to 0. 
\end{align*}
Thus, $D u= \partial u$ and $u\in\Ct^1(\Dom;\bL^\p)$.
The lemma is concluded by induction.
\end{proof}

Let ${\vf} = \tilde{\vf}/ \int_{\R^n} \tilde{\vf}$ with $\tilde{\vf}(x) := \me^{-1/(1-|x|^2)} \bm{1}_{B_1}\!(x)$ for $x\in\R^n$.
Define 
$\vf^\eps = \eps^{-n} \vf(x/\eps)$ with $\eps>0$.
It is easily seen that
\begin{gather*}
|D^m\vf^\eps|\le C \eps^{-n-m}, 
\quad
\int_{\R^n} |D^m\vf^\eps|^\p \le C\,\eps^{-(\p-1)n-\p m}
\end{gather*}
for all $m\in\Nat$ and $\p\ge 2$.
Now we mollify a function $u\in \Ct(\R^n;\bL^\p)$ using $\vf^\eps$:
\begin{equation}\label{mollify}
u^\eps := \vf^\eps * u = \int \vf^\eps(\cdot - y) u(y) \vd y.
\end{equation}
It is easily seen from Fubini's theorem that $u^\eps\in \Ct(\R^n;\bL^\p)$.
In view of Theorem~34.B in \cite{halmos1974measure},
$u(x,\om)$ is a measurable function in $x$ for each $\om$, 
and $u(\cdot,\om) \in L^\p_\loc(\R^n)$ for almost every $\om$.
Thus, by the property of mollifiers, $u^{\eps}(\cdot,\om) \in \Ct^\infty(\R^n)$ for almost every $\om$.

\begin{lemma}\label{lemA2}
If $u\in \Ct(\R^n;\bL^\p)$, then $u^\eps \in \bigcap_{m\in\Nat} \Ct^m(\R^n;\bL^\p)$, and it restricted on any $B_R$ belongs to $L^{\p}(\PS;\Ct^k(B_R))$ for all $k\in \Nat$.
\end{lemma}

\begin{proof}
By H\"older's inequality and Fubini's theorem we have
\begin{align*}
& \E |\partial^m u^\eps(x)|^\p 
= \E\left|\int_{B_\eps(x)} \partial^m\vf^\eps(x - y) u(y) \vd y\right|^\p\\
& \le |B_\eps|^{\p-1}\int_{B_\eps} |\partial^m\vf^\eps(y)|^\p  \vd y \Big(\sup_{x\in\R^n} \E|u(x)|^\p\Big)
\le C \eps^{-p(m+1)} \norm{u}_{0;\R^n}^\p,
\end{align*}
which implies that $u^\eps \in \bigcap_{m\in\Nat} \Ct^m(\R^n;\bL^\p)$ and also that $u^\eps\in L^{\p}(\PS;W^{k,\p}_\loc(\R^n))$ for any $k\in\Nat$. 
By Sobolev's embedding theorem, $u^\eps\in L^{\p}(\PS;C^{k}(B_R))$ for any $k\in\Nat$ and $R>0$.
\end{proof}

\begin{remark}\label{remA1}
If $u$ also depends on time, say $u\in \Ct(\R^n\times\R;\bL^\p)$, then the above lemma implies that $u^\eps(\cdot,t)\in L^{\p}(\PS;\Ct^k(B_R))$ and
$\sup_{t\in [-R,R]} \E |u^\eps|_{k;B_R}^\p < \infty$ for all $k\in\Nat$ and $R>0$.
In particular, $u^\eps \in L^\p_\om L^2_t H^k_x (Q_R)\cap C^k_x(Q_R;\bL^\p)$ for all $k\in\Nat$ and $R>0$.
\end{remark}

\begin{lemma}\label{lemA3}
If $u\in \Ct(\R^n;\bL^\p)$, then
\begin{equation}\label{conv}
\lim_{\eps\to 0} \E [ |u^\eps(x) - u(x)|^\p] = 0
\quad \forall\,x\in\R^n;
\end{equation}
if, in addition, $u$ is uniformly strongly continuous, namely
\[
\lim_{\eps\to 0}\sup_{x\in\R^n}\max_{|y|\le\eps} \E[| u(x+y)- u(x)|^\p] = 0,
\]
then the convergence \eqref{conv}
 is uniform with respect to $x\in\R^n$.
\end{lemma}

\begin{proof}
Using the continuity of $u$ we have
\begin{align*}
\E |u^\eps(x) & -  u(x)|^\p 
 = \E\left|\int_{B_\eps(x)} \vf^\eps(x - y) [ u(y) - u(x)] \vd y\right|^\p\\
& \le |B_\eps|^{\p-1}\int_{B_\eps} |\vf^\eps(y)|^\p \vd y 
\, \Big(\max_{|y|\le\eps} \E| u(x+y) -  u(x)|^\p\Big) \\
& \le C\max_{|y|\le\eps} \E| u(x+y)- u(x)|^\p \to 0
\quad
\text{as}~ \eps \to 0.
\end{align*}
Then the lemma is easily concluded.
\end{proof}

\begin{lemma}\label{lemA4}
If $u\in \Ct^1(\R^n;\bL^\p)$,
then $u^\eps \in \Ct^1(\R^n;\bL^\p)$
and $Du^\eps = \vf^\eps * Du$.
\end{lemma}

\begin{proof}
We compute that
\begin{align*}
& \E \left| r^{-1}[u^\eps(x+re_i) - u^\eps(x)] - \vf^\eps * D_i u(x) \right|^\p \\
& = \E\left|\int_{B_\eps(x)}  \vf^\eps(x - y) \left(\frac{u(y+re_i) - u(y)}{r} - D_i u(y)\right) \vd y\right|^\p\\
& \le \int_{B_\eps(x)} \int_0^1 \vf^\eps(x - y) \,\E\big| D_i u (y+sre_i) - D_i u(y) \big|^\p \vd s \md y.
\end{align*}
Since $\lim_{r\to 0} \E\big| D_i u (y+sre_i) - D_i u(y) \big|^\p
=0$ and  
$ \E\big| D_i u (y+sre_i) - D_i u(y) \big|^\p
\le 2^\p \norm{u}_{1;\R^n}^\p$,
the lemma is concluded by the dominated convergence theorem.
\end{proof}

The following consequence is straightforward.

\begin{corollary}\label{corA5}
If $u\in \Ct^{m}(\R^n;\bL^\p)$,
then $D^m u^\eps(x)$
converges to $D^m u(x)$ in $\bL^\p$ for each $x\in\R^n$, as $\eps$ tends to zero;
if, in addition, $D^m u$ is uniformly strongly continuous, then the convergence is uniform with respect to $x\in\R^n$.
\end{corollary}

The final lemma concerns the convergence of H\"older norms.

\begin{lemma}\label{lemA1.5}
If $u\in \Ct^\al(\R^n;\bL^\p)$ with $\al\in(0,1)$, then $u^\eps \in\Ct^\al(\R^n;\bL^\p)$ satisfying $\hnorm{u^\eps}_{\alpha;\R^n} \le C(n,\p)\hnorm{u}_{\alpha;\R^n}$
and $\lim_{\eps\to 0}\hnorm{u^\eps - u}_{\alpha/2;\R^n}=0$.
\end{lemma}

\begin{proof}
By H\"older's inequality and Fubini's theorem we have
\begin{align*}
& \E |[u^\eps(x) - u^\eps(x')|^\p 
 = \E\left|\int_{B_\eps} \vf^\eps(y) [u(x-y)-u(x'-y)] \vd y\right|^\p\\
& \le |B_\eps|^{\p-1}\int_{B_\eps} |\vf^\eps(y)|^\p  \vd y \Big(\sup_{y\in B_\eps} \E|u(x-y)-u(x'-y)|^\p\Big)
\le C \hnorm{u}_{\alpha;\R^n}^\p |x-x'|^{\al\p},
\end{align*}
thus $\hnorm{u^\eps}_{\alpha;\R^n} \le C\hnorm{u}_{\alpha;\R^n}$. 
Furthermore, by the definition of H\"older norms we can easily seen that 
\[
\hnorm{u^\eps - u}_{\alpha/2;\R^n}^2 \le \hnorm{u^\eps - u}_{\alpha;\R^n}
\norm{u^\eps - u}_{0;\R^n}\le C  \hnorm{u}_{\alpha;\R^n} \norm{u^\eps - u}_{0;\R^n},
\]
so the proof is concluded by Lemma~\ref{lemA3}. 
\end{proof}

\bibliographystyle{amsalpha}
\small{
\bibliography{ref_spde}
}

\end{document}